\documentclass{article}
\usepackage[a4paper]{geometry}
\usepackage{amsrefs,amssymb,amsthm,color,mathrsfs,mleftright,tikz-cd}
\usepackage[fleqn]{amsmath}

\newtheorem{theorem}{Theorem}[section]
\newtheorem{lemma}[theorem]{Lemma}
\newtheorem{corollary}[theorem]{Corollary}

\theoremstyle{definition}
\newtheorem{definition}{Definition}

\tikzcdset{arrow style=tikz, diagrams={>=stealth}}
\tikzcdset{every label/.append style = {font = \normalsize}}
\definecolor{hcol}{RGB}{33,120,33}
\definecolor{fcol}{RGB}{44,90,160}
\definecolor{gcol}{RGB}{170,0,0}

\makeatletter
\renewcommand\section{\@startsection {section}{1}{\z@}%
                                   {-3.5ex \@plus -1ex \@minus -.2ex}%
                                   {1.3ex \@plus.2ex}%
                                   {\bf\large}}
\makeatother

\numberwithin{equation}{section}
\numberwithin{figure}{section}

\renewcommand{\setminus}{ - }
\DeclareMathOperator{\fix}{\textnormal{fix}}
\DeclareMathOperator{\diam}{\textnormal{diam}}
\title{ \vspace{-5ex}\bf \Large Forward limit sets of semigroups of substitutions}

\author{Ibai Aedo, Uwe Grimm, and Ian Short}

\date{\vspace{-5ex}}

\begin{document}

\maketitle

{\centering\small Dedicated to Uwe Grimm (1963--2021).\par}

\begin{abstract}
We introduce the forward limit set $\Lambda$ of a semigroup $S$ generated by a family of substitutions of a finite alphabet, which typically coincides with the set of all possible s-adic limits of that family. We provide several alternative characterisations of the forward limit set. For instance, we prove that $\Lambda$ is the unique maximal closed and strongly $S$-invariant subset of the space of all infinite words, and we prove that it is the closure of the set of images of all fixed points under $S$. It is usually difficult to compute a forward limit set explicitly; however, we show that, provided certain assumptions hold, $\Lambda$ is uncountable, and we supply upper bounds on its size in terms of logarithmic Hausdorff dimension.
\end{abstract}

\section{Introduction}\label{secA}

The objective of this paper is to characterise forward limit sets of semigroups of substitutions acting on finite alphabets and to determine the size of such limit sets. The work is motivated by the theory of s-adic sequences, which is concerned with the dynamics of compositions of sequences of substitutions. Here we take a global view of such systems, examining the full collection of all s-adic sequences corresponding to some specified finite set of substitutions and considering the relationship of these sequences to the semigroup generated by the substitutions. This line of enquiry is inspired by semigroup theory in other disciplines, such as complex dynamics, hyperbolic dynamics, and hyperbolic geometry \cite{AvBoYo2010,FrMaSt2012,HiMa1996,JaSh2022}, where the relationship between semigroups and s-adic type sequences is explored, and where forward limit sets are of foundational importance.

Let $\mathscr{F}$ be a collection of substitutions defined over a finite alphabet $\mathscr{A}$. We denote by $\mathscr{A}^+$ the collection of finite words over $\mathscr{A}$, and we denote the collection of infinite words over $\mathscr{A}$ by $\mathscr{A}^\mathbb{N}$ (where $\mathbb{N}=\{1,2,\dotsc\}$). By composing elements of $\mathscr{F}$ we obtain a semigroup $S$, which we call the \emph{substitution semigroup} generated by $\mathscr{F}$. This acts on both $\mathscr{A}^+$ and $\mathscr{A}^\mathbb{N}$, and on their union $\widetilde{\mathscr{A}}=\mathscr{A}^+\cup\mathscr{A}^\mathbb{N}$, in the usual way. 

The set $\widetilde{\mathscr{A}}$ is a complete, compact metric space with the word metric. We write $S(X)$ for the set $\{s(x): s\in S,\, x\in X\}$, and we denote the closure of a set $Y$ in $\widetilde{\mathscr{A}}$ by $\overline{Y}$. This brings us to the central object of interest.

\begin{definition}
The \emph{forward limit set} $\Lambda(A)$ of the subset $A$ of $\mathscr{A}$ for the substitution semigroup $S$ is the complement of $S(A)$ in $\overline{S(A)}$; that is,  
\[
\Lambda(A)=\overline{S(A)}\setminus S(A).
\]
The \emph{forward limit set} of $S$ is the set $\Lambda(\mathscr{A})$. We denote this set by $\Lambda$.
\end{definition}

For example, consider the substitution semigroup $S$ generated by the pair of substitutions
\[
f\colon\
\begin{matrix}
a &\longmapsto &ab\\
b &\longmapsto &ba
\end{matrix}
\qquad\text{and}\qquad
g\colon\
\begin{matrix}
a& \longmapsto &ba\\
b& \longmapsto &ab
\end{matrix}\,.
\]
The substitution $f$ is known as the \emph{Thue--Morse substitution} \cite[Section~4.6]{BaGr2013} and it has two fixed points, both aperiodic, namely $x=abbabaa\dotsc$ and $h(x)$, where $h(a)=b$ and $h(b)=a$. Observe that $g=h\circ f=f\circ h$, so $S=\{f^n:n\in\mathbb{N}\}\cup \{h\circ f^n:n\in\mathbb{N}\}$. It follows that $\Lambda=\{x,h(x)\}$. 

The actions of $f$ and $g$ can be illustrated, in part, by the \emph{first-letter graph} for $f$ and $g$ shown in Figure~\ref{figA}, which uses arrows between the vertices $a$ and $b$ to represent the first letters of the words $f(a)$, $f(b)$, $g(a)$, and $g(b)$ (first-letter graphs will be defined formally in Section~\ref{secB}).

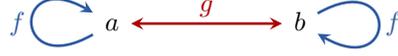
\begin{figure}[ht]
\begin{centering}
\begin{tikzcd}[column sep=2cm]
a
\ar[loop, line width=1, fcol, out=210, in=150, distance=35,"f"]
\ar[r, line width=1, <->,gcol,"g"]
&b
\ar[loop, line width=1, fcol, out=30, in=330, distance=35,"f"]
\end{tikzcd}
\caption{First-letter graph for the Thue--Morse substitutions}
\label{figA}
\end{centering}
\end{figure}

Let  us consider another example of a forward limit set, this time one associated to the substitution semigroup $S$ generated by the triple of substitutions
\[
f\colon\
\begin{matrix}
a & \longmapsto & ab\\
b &\longmapsto &a\phantom{a}
\end{matrix}\,,\qquad
g\colon\
\begin{matrix}
a &\longmapsto &ba\\
b &\longmapsto &a\phantom{a}
\end{matrix}\,,
\qquad
h\colon\
\begin{matrix}
a &\longmapsto &b\\
b &\longmapsto &a
\end{matrix}\,.
\]
This triple is illustrated in Figure~\ref{figB}. The substitution $f$ is known as the \emph{Fibonacci substitution} \cite[Example~4.6]{BaGr2013}, and it is known that $S$ is the collection of all \emph{Sturmian substitutions} (see \cite[Chapters 6 \& 9]{Fo2002} or \cite[Chapter 2]{Lo2002}). In this case, the forward limit set of $S$ is the set $B(a,b)$ of all balanced infinite words over $\{a,b\}$; these are infinite words $w$ over $\{a,b\}$ with the property that the number of $a$s (or $b$s) in any two (finite) subwords of $w$ of the same length differs by at most one. That $\Lambda=B(a,b)$ can be deduced from results such as \cite[Proposition 4.8]{Ri2021}. We omit the details, however, we will at least demonstrate shortly that the forward limit set of $S$ is uncountable.

\begin{figure}[ht]
\begin{centering}
\begin{tikzcd}[column sep=25mm]
a
\ar[loop, line width=1, fcol, out=210, in=150, distance=35,"f"]
\ar[r, line width=1, gcol, <->, bend right=45,"g"]
\ar[r, line width=1, hcol, <->,"h"]
&b
\ar[l, line width=1, fcol, bend right=45, "f",swap]
\end{tikzcd}
\caption{First-letter graph for the Sturmian substitutions}
\label{figB}
\end{centering}
\end{figure}
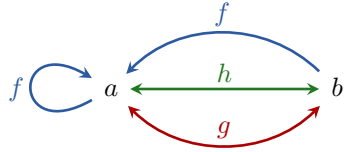

The forward limit set $\Lambda(A)$ is (by definition) the boundary of $S(A)$, and it lies within $\mathscr{A}^\mathbb{N}$ because $\mathscr{A}^+$ has the discrete topology (so $\overline{S(A)}$ does not accumulate in $\mathscr{A}^+$). The set $\Lambda(A)$ is closed, and we will see later (Lemma~\ref{lemS}) that it is $S$-invariant, in the sense that each member of $S$ maps $\Lambda(A)$ into itself. We explore further basic properties of forward limit sets in Section~\ref{secD}. 

We will usually be concerned with \emph{fixed-letter-free} substitution semigroups; a semigroup $S$ is of this type if each substitution $f$ in $S$ satisfies $f(a)\neq a$, for all  $a\in\mathscr{A}$. With this assumption we can relate forward limit sets to s-adic limits. We recall (from, for example, \cite[Section~4.11]{BeRi2010}) that an infinite word $x$ over $\mathscr{A}$ is an \emph{s-adic limit} of a family of substitutions $\mathscr{F}$ if there exist sequences $(f_n)$ in $\mathscr{F}$ and $(a_n)$ in $\mathscr{A}$ with $f_1\circ f_2\circ \dotsb \circ f_n(a_n)\to x$ as $n\to\infty$. 

\begin{theorem}\label{thmA}
Let $S$ be a fixed-letter-free substitution semigroup with finite generating set $\mathscr{F}$. The forward limit set of $S$ is equal to the set of all s-adic limits of $\mathscr{F}$.
\end{theorem}

Theorem~\ref{thmA} will be proved in Section~\ref{secE}. The theorem fails if $S$ has fixed letters. For example, consider the semigroup generated by substitutions $f$ and $g$ of a two-letter alphabet $\{a,b\}$, where $f(a)=g(a)=a$, $f(b)=ab$ and $g(b)=bb$. Then $g^n\circ f(b)\to x$, where $x=abbb\dotsc$ (and where $g^n=g\circ g\circ \dotsb \circ g$, the $n$-fold composition of $g$ with itself). However, $x$ is not an s-adic limit of $\mathscr{F}$. This can be seen from a short argument by contradiction. We assume that $f_1\circ f_2\circ \dotsb \circ f_n(a_n)\to x$, where $f_i\in\{f,g\}$ and $a_i\in\{a,b\}$. Then $a_n=b$ for all but finitely many values of $n$, and the fact that there is only a single occurence of $a$ in $x$ implies that none of the substitutions $f_i$ can equal $f$, which gives the required contradiction.

Next we turn to classifying forward limit sets by means of the action of $S$ on $\mathscr{A}^\mathbb{N}$. A subset $X$ of $\mathscr{A}^\mathbb{N}$ is said to be \emph{strongly $S$-invariant} if $S(X)=X$; that is, $X$ is strongly $S$-invariant if it is $S$-invariant and $S$ maps $X$ onto $X$.

\begin{theorem}\label{thmB}
Let $S$ be a finitely-generated fixed-letter-free substitution semigroup. A subset of $\mathscr{A}^{\mathbb{N}}$ is closed and strongly $S$-invariant if and only if it is the forward limit set of some subset of $\mathscr{A}$ for $S$. 
\end{theorem}

Theorem~\ref{thmB} will be proved in Section~\ref{secF} along with various related results. A consequence of this theorem is another characterisation of $\Lambda$, which says that $\Lambda$ is the greatest element in the poset of closed and strongly $S$-invariant subsets of $\mathscr{A}^{\mathbb{N}}$.

\begin{corollary}\label{corA}
Let $S$ be a fixed-letter-free substitution semigroup with finite generating set $\mathscr{F}$. Every closed and strongly $S$-invariant subset of $\mathscr{A}^{\mathbb{N}}$ is contained in the forward limit set of~$S$.
\end{corollary}

In other branches of dynamics, forward limit sets and related sets can often be characterised in terms of fixed points. In this work, a \emph{fixed point} of a substitution $f$ is a word $x$ (finite or infinite) with $f(x)=x$. We denote by $\fix(S)$ the collection of all fixed points of substitutions from $S$.

\begin{theorem}\label{thmD}
The forward limit set of a finitely-generated fixed-letter-free substitution semigroup $S$ is equal to $\overline{S(X)}$, where $X=\fix(S)$.
\end{theorem}

Theorem~\ref{thmD} will be proved in Section~\ref{secG}.

The results stated so far concern characterisations of forward limit sets. The next two results are about the size of forward limit sets. The first one (proved in Section~\ref{secH}) says that, under certain assumptions, the forward limit set of a substitution semigroup is uncountable. In stating this result, we describe one word $a$ as a \emph{prefix} of another word $b$ if $b=ac$, where $c$ is either another word or the empty word. 

\begin{theorem}\label{thmE}
Let $S$ be a fixed-letter-free substitution semigroup with finite generating set $\mathscr{F}$. Suppose that $\fix(S)$ contains two words with the same first letter and suppose also that $f(a)$ is not a prefix of $f(b)$, for any distinct
$a,b\in\mathscr{A}$ and $f\in\mathscr{F}$. Then the forward limit set of $S$ is uncountable.
\end{theorem}

The scope of application of Theorem~\ref{thmE} is broader than it might first appear. For example, the Fibonacci substitutions $f$ and $g$ illustrated in Figure~\ref{figB} do not satisfy the hypotheses of the theorem, because $f$ fails the prefix property (since $f(b)=a$ is a prefix of $f(a)=ab$). However, the pair of substitutions $g$ and $k=f\circ g$ do satisfy the hypotheses, from which it follows that the semigroup generated by this pair has an uncountable forward limit set, so the forward limit set of the semigroup generated by $f$ and $g$ is also uncountable. Consequently,  the forward limit set of the semigroup of all Sturmian substitutions, which is generated by $f$, $g$ and $h$, where $h(a)=b$ and $h(b)=a$, is uncountable too.

The hypothesis that $\fix(S)$ contains two words with the same first letter rules out semigroups  generated by a single substitution -- such semigroups have finite forward limit sets. The prefix hypothesis rules out semigroups generated by substitutions that are not injective. For example, consider the substitution semigroup $S$ of the two-letter alphabet $\{a,b\}$ generated by substitutions $f$ and $g$, where $f(a)=f(b)=aa$ and $g(a)=g(b)=ab$. The forward limit set of $S$ consists of the two fixed points of $f$ and $g$.

There are more complex substitution semigroups with countable forward limit sets. Consider, for example, the semigroup $S$ generated by the substitutions
\[
f\colon\
\begin{matrix}
a& \longmapsto &ac\phantom{a}\\
b& \longmapsto &cb\phantom{a}\\
c& \longmapsto &cba
\end{matrix}
\qquad\text{and}\qquad
g\colon
\begin{matrix}
a& \longmapsto &bac\\
b& \longmapsto &c\phantom{aa}\\
c& \longmapsto &cba
\end{matrix}\,,
\]
illustrated in Figure~\ref{figD}, which fails the prefix condition of Theorem~\ref{thmE}.

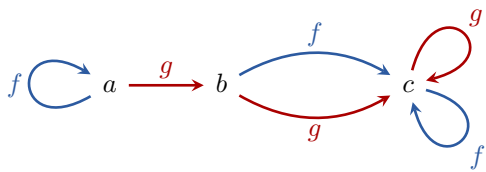
\begin{figure}[ht]
\begin{centering}
\begin{tikzcd}[column sep=1cm]
a
\ar[loop, line width=1, fcol, out=210, in=150, distance=35,"f"]
\ar[r, line width=1, gcol,"g"]
&b
\ar[rr, line width=1, gcol, bend right,"g",swap]
\ar[rr, line width=1, fcol, bend left,"f"]
&&c
\ar[loop, line width=1, fcol, out=345, in=285, distance=35,"f"]
\ar[loop, line width=1, gcol, out=75, in=15, distance=35,"g"]
\end{tikzcd}
\caption{First-letter graph for $f$ and $g$}
\label{figD}
\end{centering}
\end{figure}

One can check that $f\circ g =g^2$, and from this it follows that the semigroup $S$ generated by $\mathscr{F}=\{f,g\}$ comprises all substitutions of the form $g^m\circ f^n$, where $m$ and $n$ are positive integers or $0$ (and they are not both $0$). Then it can be shown that $\Lambda$ consists of the fixed point $x$ of $f$ with first letter $c$ (also a fixed point of $g$) as well as all points $g^m(y)$, for $m=0,1,2,\dotsc$, where $y$ is the fixed point of $f$ with first letter $a$. Thus $\Lambda$ is infinite and countable.

Informally speaking, Theorem~\ref{thmE} shows that (typically) forward limit sets cannot be too small; the next theorem says that forward limit sets cannot be too large. It is framed using logarithmic Hausdorff dimension, which can be used to quantify the size of sets for which the usual Hausdorff dimension is zero. 

In this theorem the \emph{length} of a finite set of substitutions $\mathscr{F}$ is the least length (number of letters) of $f(a)$, among all letters $a\in\mathscr{A}$ and substitutions $f\in\mathscr{F}$. 

\begin{theorem}\label{thmF}
Let $S$ be a substitution semigroup generated by a finite set $\mathscr{F}$ of substitutions of length $r$ and size $s$, where $r>1$. Then the forward limit set of $S$ has logarithmic Hausdorff dimension at most $\log_rs$.

Furthermore, given any pair of positive integers $r$ and $s$ with $r>1$ and $s\leqslant|\mathscr{A}|^{r-1}$, there exists a substitution semigroup $S$ with these parameters for which the bound $\log_rs$ is attained.
\end{theorem}

Theorem~\ref{thmF} will be proved in Section~\ref{secI}. 

The assumption that $r>1$ is not as prohibitive as it may seem, for fixed-letter-free substitution semigroups at least. To justify this claim, briefly, consider once more the Fibonacci substitutions $f$ and $g$ illustrated in Figure~\ref{figB}. This pair of substitutions has length 1. However, the semigroup generated by $f$ and $g$ has the same forward limit set $\Lambda$ as the similar semigroup generated by the collection $\{f^2,f\circ g,g\circ f,g^2\}$, which has length 2 (and size 4). Theorem~\ref{thmF} then tells us that the logarithmic Hausdorff dimension of $\Lambda$ is less than or equal to $\log_24=2$. 

In the final section (Section~\ref{secJ}) we introduce the hull of the forward limit set of a substitution semigroup $S$ and prove that, with some assumptions, it is the least element in the poset of closed, $S$-invariant and shift invariant non-empty subsets of $\mathscr{A}^\mathbb{N}$. 

\section{Substitution semigroups}\label{secB}

In this section we expand on the theory of substitution semigroups summarised in the introduction. Supporting material on the theory of substitutions can be found in, for example, \cite{BaGr2013,BeRi2010,Fo2002}.

Throughout this work $\mathscr{A}$ will denote a finite set of size at least two, called an \emph{alphabet}. The elements of $\mathscr{A}$ are called \emph{letters} of $\mathscr{A}$. A \emph{word} over $\mathscr{A}$ is a finite sequence of letters of $\mathscr{A}$. We write the word $a_1,a_2,\dots,a_n$ as a string $a_1a_2\dotsc a_n$. The \emph{length} of the word $w=a_1a_2\dotsc a_n$, denoted $|w|$, is the length $n$ of the sequence.  The collection of all words over $\mathscr{A}$ forms a semigroup $\mathscr{A}^+$ with composition given by concatenation of words. We exclude the empty word from $\mathscr{A}^+$.

An \emph{infinite word} over $\mathscr{A}$ is an infinite sequence of elements of $\mathscr{A}$. The collection of all infinite words over $\mathscr{A}$ is denoted $\mathscr{A}^{\mathbb{N}}$, and a typical element of this set is written in the form $a_1a_2\dotsc$, for letters $a_i$. If $w_1, w_2,\dotsc$ are words over $\mathscr{A}$, then $w_1w_2\dotsc$ denotes the infinite word over $\mathscr{A}^+$ with the obvious meaning. Note that we only consider one-sided infinite words. We define $\widetilde{\mathscr{A}} = \mathscr{A}^+\cup\mathscr{A}^{\mathbb{N}}$, the collection of all words and infinite words over $\mathscr{A}$. Given any word $w$ of $\widetilde{\mathscr{A}}$ we write $\pi_k(w)$ for the $k$th letter in the sequence $w$, providing that $w$ has length at least $k$. If $w$ has length less than $k$, then we define $\pi_k(w)=\alpha$, where $\alpha$ is some additional letter not in $\mathscr{A}$ (this additional letter is only needed for the definition of the metric, to follow). We define a metric $d$ on $\widetilde{\mathscr{A}}$ by the formulas $d(u,u)=0$ and 
\[
d(u,v) = \frac{1}{2^{n-1}},\quad\text{where $n=\min\{k\in\mathbb{N}: \pi_k(u)\neq \pi_k(v)\}$,}
\]
for distinct elements $u$ and $v$ of $\widetilde{\mathscr{A}}$. Then $(\widetilde{\mathscr{A}},d)$ is a complete, compact metric space, and the associated topology is the product topology.

A \emph{substitution} of $\mathscr{A}$ is a function $f\colon \mathscr{A}^+\longrightarrow \mathscr{A}^+$ with the property \(f(uv)=f(u)f(v)\), for any two words $u$ and $v$ in $\mathscr{A}^+$. Substitutions are sometimes called \emph{non-erasing morphisms} (and other phrases are used too). A substitution is specified uniquely by the images of the letters of the alphabet. Let $\mathscr{F}$  denote a non-empty finite collection of substitutions of $\mathscr{A}$. We can form words over $\mathscr{F}$, just as we formed words over $\mathscr{A}$, giving rise to a semigroup $\mathscr{F}^+$ with composition given by concatenation of words. This is not the semigroup of primary interest though; our main focus is the semigroup obtained from $\mathscr{F}$ through composition of functions.

In more detail, given two substitutions $f$ and $g$ of $\mathscr{A}$, we denote by $f\circ g$ the functional composition of first $g$ and then $f$. A \emph{substitution semigroup} is a set of substitutions of the same alphabet that is closed under composition. We say that a substitution semigroup $S$ is \emph{generated} by a collection of substitutions $\mathscr{F}$ if each element of $S$ is a finite composition of functions from $\mathscr{F}$. We describe $\mathscr{F}$ as a \emph{generating set} for $S$. The substitution semigroup $S$ is said to be \emph{finitely generated} if it has a finite generating set.

We have defined $S$ by its action on $\mathscr{A}^+$; there are also two other actions of $S$ of interest to us. The first is on the collection of infinite words $\mathscr{A}^\mathbb{N}$. Given a substitution $f$ in $S$ and an infinite word $a_1a_2\dotsc$, where $a_i\in\mathscr{A}$, we define 
\[
f(a_1a_2\dotsc)=f(a_1)f(a_2)\dotsc.
\]
As promised, this defines an action of the semigroup $S$ on $\mathscr{A}^{\mathbb{N}}$. Notice that if $f\in S$, and if $u,v\in\widetilde{\mathscr{A}}$ share the same first letter (that is, $\pi_1(u)=\pi_1(v)$), then
\[
d(f(u),f(v)) \leqslant  d(u,v).
\]
Hence the semigroup action is continuous.

The second action of $S$ that we consider is an action on the alphabet $\mathscr{A}$ itself. For $f\in \mathscr{F}$ and $a\in\mathscr{A}$, we define
\[
f[a] = \pi_1(f(a)).
\]
Thus $f[a]$ is the first letter of $f(a)$. The square brackets in $f[a]$ distinguish this single letter $f[a]$ from the word $f(a)$. This action on $\mathscr{A}$ can be visualised using the finite directed graph $G_{\mathscr{F}}$ with vertex set $\mathscr{A}$ and, for each $f\in \mathscr{F}$ and $a\in\mathscr{A}$, a directed edge from $a$ to $f[a]$. Note that $G_{\mathscr{F}}$ may include multiple edges between any pair of vertices and edges from any vertex to itself (loops). We call $G_{\mathscr{F}}$ the \emph{first-letter graph} for $\mathscr{F}$. This graph is used to frame some of our later results.

For example, consider the semigroup $S$ generated by the three substitutions
\[
f\colon\
\begin{matrix}
a&\longmapsto&ea\\
b&\longmapsto&de\\
c&\longmapsto&ce\\
d&\longmapsto&da\\
e&\longmapsto&eb
\end{matrix}\,,
\qquad
g\colon\
\begin{matrix}
a&\longmapsto&bc\\
b&\longmapsto&bd\\
c&\longmapsto&ce\\
d&\longmapsto&db\\
e&\longmapsto&ec
\end{matrix}\,,
\qquad
h\colon\
\begin{matrix}
a&\longmapsto&ca\\
b&\longmapsto&cb\\
c&\longmapsto&cc\\
d&\longmapsto&ed\\
e&\longmapsto&dc
\end{matrix}\,.
\]

The first-letter graph for $\mathscr{F}=\{f,g,h\}$ is shown in Figure~\ref{figC}, with labels for loops omitted.

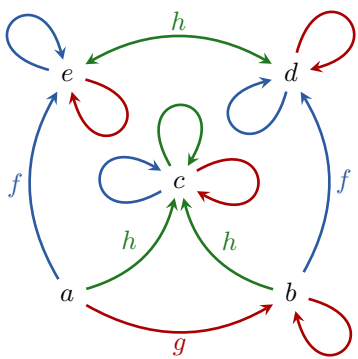
\begin{figure}[ht]
\begin{centering}
\begin{tikzcd}[column sep=1cm, row sep=1cm]
e
\ar[rr, leftrightarrow, line width=1, hcol, bend left,"h"]
\ar[loop, line width=1, gcol, out=345, in=285, distance=35]
\ar[loop, line width=1, fcol, out=165, in=105, distance=35]
&&d
\ar[loop, line width=1, gcol, out=75, in=15, distance=35]
\ar[loop, line width=1, fcol, out=255, in=195, distance=35]
\\
&c
\ar[loop, line width=1, hcol, out=120, in=60, distance=35]
\ar[loop, line width=1, gcol, out=30, in=330, distance=35]
\ar[loop, line width=1, fcol, out=210, in=150, distance=35]
&
\\
a
\ar[ru, line width=1, hcol, bend right,"h"]
\ar[uu, line width=1, fcol, bend left,"f"]
\ar[rr, line width=1, gcol, bend right,"g",swap]
&& b
\ar[lu, line width=1, hcol, bend left,"h", swap]
\ar[uu, line width=1, fcol, bend right,"f",swap]
\ar[loop, line width=1, gcol, out=345, in=285, distance=35]
\end{tikzcd}
\caption{First-letter graph for $\mathscr{F}=\{f,g,h\}$}
\label{figC}
\end{centering}
\end{figure}

Two vertices $a$ and $b$ of a first-letter graph $G_{\mathscr{F}}$ (for some finite set of substitutions $\mathscr{F}$) are said to be \emph{strongly connected} if there is a directed walk in $G_{\mathscr{F}}$ from $a$ to $b$ and another directed walk from $b$ to $a$. In other words, $a$ and $b$ are strongly connected if there exists $f,g\in S$ with $f[a]=b$ and $g[b]=a$. Let $\mathscr{A}_0$ denote the subset of $\mathscr{A}$ comprising those vertices  of $G_{\mathscr{F}}$ that are strongly connected to themselves. The property of being strongly connected is an equivalence relation on $\mathscr{A}_0$; the equivalence classes are called \emph{strongly connected components} of $G_{\mathscr{F}}$. 

A \emph{terminal component} of $G_{\mathscr{F}}$ is a strongly connected component $A$ for which every walk with initial vertex in $A$ also has final vertex in $A$. In other words, a terminal component is a strongly connected component $A$ with the property that if $a\in A$ and $g\in S$, then the vertex $g[a]$ also belongs to $A$.

For example, in Figure~\ref{figC} we have $\mathscr{A}_0=\{b,c,d,e\}$ and the strongly connected components of $G_{\mathscr{F}}$ are $\{b\}$, $\{c\}$, and $\{d,e\}$. The terminal components are $\{c\}$ and $\{d,e\}$.

\section{Fixed-letter-free substitution semigroups}\label{secC}

In this section we prove that the fixed-letter-free property for a substitution semigroup is equivalent to a growth condition for images of letters under the semigroup (Theorem~\ref{thmG}). 

We begin with a lemma that has no doubt been proved many times before; however, the proof is sufficiently short and the result sufficiently important that it merits inclusion.

%Later on we will apply this lemma to the alphabet $\mathscr{A}$ and to the alphabet $\mathscr{F}$. For this reason, we state the lemma using a different notation for the alphabet altogether.

\begin{lemma}\label{lemA}
Let $\mathscr{X}$ be a finite alphabet and let $(u_n)$ be an infinite sequence of words over $\mathscr{X}$ for which $|u_n|\to\infty$ as $n\to\infty$. Then there is a subsequence $(v_n)$ of $(u_n)$ of the form
\[
v_n=x_1x_2\dotsc x_n w_n,\quad n=1,2,\dotsc,
\]
where $x_1,x_2,\dotsc$ are letters of $\mathscr{X}$ and $w_1,w_2,\dotsc$ are words over $\mathscr{X}$. 
\end{lemma}
\begin{proof}
We will define a chain $I_0\supseteq I_1\supseteq I_2 \supseteq \dotsb$ of infinite subsets of $\mathbb{N}$ (where $I_0=\mathbb{N}$) and a sequence of letters $x_1,x_2,\dotsc$ of $\mathscr{X}$ with the property that, for each $k\in I_n$, we have $|u_k|> n$ and $\pi_i(u_k)=x_i$, for $i=1,2,\dots,n$.

To do this, suppose that $I_0,I_1,\dots,I_{n-1}$ and $x_1,x_2,\dots,x_{n-1}$ have been constructed with this property, where $n\geqslant 1$. Let $I_{n-1}^*$ be the infinite subset of $I_{n-1}$ of those integers $k$ for which $|u_k|> n$. Now, there are only finitely many choices for the letter $\pi_n(u_k)$, for $k\in I_{n-1}^*$. Since $I_{n-1}^*$ is infinite, we can choose an infinite subset $I_n$ of $I_{n-1}^*$ and a letter $x_n$ of $\mathscr{X}$ for which $\pi_n(u_k)=x_n$, for all $k\in I_n$. Thus $I_0,I_1,\dots,I_{n}$ and $x_1,x_2,\dots,x_{n}$ have the required property, and the existence of the full chain $ I_1, I_2, \dotsc$  and sequence $x_1,x_2,\dotsc$ follows from the axiom of dependent choices.

We now choose $i_n\in I_n$, for $n=1,2,\dotsc$, and define $v_n=u_{i_n}$. This sequence $(v_n)$ is of the form specified in the lemma.
\end{proof}

Sometimes we apply the reverse form of Lemma~\ref{lemA}, in which $v_n=w_n x_n x_{n-1}\dotsc x_1$, for $x_i\in\mathscr{X}$ and $w_n\in \mathscr{X}^+$. This can be obtained quickly from Lemma~\ref{lemA} by judicious use of the operation that reverses the letters of a word.

%Recall that a word $w$ over $\mathscr{A}$ is said to be fixed by a substitution $f$ if $f(w)=w$. The next lemma tells us if a substitution fixes a word over $\mathscr{A}$, then it must fix a letter in $\mathscr{A}$.

%\begin{lemma}\label{lemC}
%Let $f$ be a substitution of an alphabet $\mathscr{A}$ and let $u$ be a word over $\mathscr{A}$ with first letter $a$. Then $\pi_1(f(a))=\pi_1(f(u))$.
%\end{lemma}

The next lemma provides conditions under which a substitution must have a fixed letter.

\begin{lemma}\label{lemB}
Let $f$ be a substitution of an alphabet $\mathscr{A}$ and let $u$ be a word over $\mathscr{A}$ with $|f(u)|=|u|$ and $\pi_1(f(u))=\pi_1(u)$. Then $f$ fixes the letter $\pi_1(u)$. 
\end{lemma}
\begin{proof}
Let $u=a_1a_2\dotsc a_n$, where $a_i\in\mathscr{A}$. Then $a_1=\pi_1(u)$ and 
\[
|f(a_1)|+|f(a_2)|+\dotsb +|f(a_n)|=|f(u)|=|u|=n. 
\]
Hence $|f(a_1)|=|f(a_2)|=\dots = |f(a_n)|=1$; in particular, $|f(a_1)|=1$. It follows that $f(a_1)$ is a single letter. Now, $\pi_1(f(a_1))=\pi_1(f(u))$ (because $a_1=\pi_1(u)$), so we see that
\[
f(a_1)=\pi_1(f(a_1))=\pi_1(f(u))=\pi_1(u)=a_1,
\]
as required.
\end{proof}

\begin{corollary}\label{corB}
Let $f$ be a substitution of an alphabet $\mathscr{A}$. Then $f$ has a fixed letter in $\mathscr{A}$ if and only if it has a fixed word in $\mathscr{A}^+$.
\end{corollary}
\begin{proof}
Suppose that $f$ fixes a word $w$ over $\mathscr{A}$. Then $f(w)=w$. Applying Lemma~\ref{lemB}, we see that $f$ fixes the letter $\pi_1(w)$. The converse implication is immediate.
\end{proof}

We can now prove the main theorem of this section. In the proof, we use the map $F$ from $\mathscr{F}^+$ to $S$ that sends the word $f_1f_2\dotsc f_n$ in $\mathscr{F}^+$ to the substitution $f_1\circ f_2 \circ \dotsb \circ f_n$ in $S$.

\begin{theorem}\label{thmG}
Let $S$ be a substitution semigroup with finite generating set $\mathscr{F}$. Then $S$ has the fixed-letter-free property if and only if
\[
\min \{ |f_1\circ f_2\circ \dots\circ  f_n(a)| : f_i \in \mathscr{F}\} \to \infty \quad \text{as $n\to\infty$,}
\]
for each $a\in\mathscr{A}$.
\end{theorem}
\begin{proof}
Suppose first that there is a substitution $g$ in $S$ and a letter $a$ in $\mathscr{A}$ with $g(a)=a$. We can find generators $g_1,g_2,\dotsc,g_m$ in $\mathscr{F}$ with  $g=g_1\circ g_2\circ \dotsb\circ  g_m$. Then 
\[
|(g_1\circ g_2\circ \dotsb\circ  g_m)^n(a)|=|g^n(a)|=|a|=1,
\]
for all positive integers $n$. Consequently, 
\[
\min \{ |f_1\circ f_2\circ \dotsb\circ  f_n(a)| : f_i \in \mathscr{F}\} \nrightarrow \infty \quad \text{as $n\to\infty$.}
\]

Conversely, suppose that 
\[
\min \{ |f_1\circ f_2\circ \dotsb\circ  f_n(a))| : f_i \in \mathscr{F}\} \nrightarrow \infty \quad \text{as $n\to\infty$,}
\]
for some letter $a$ in $\mathscr{A}$. Then there is a positive integer $N$ and a sequence $(u_n)$ in $\mathscr{F}^+$ with $|u_1|<|u_2|<\dotsb$ and
\[
|F(u_n)(a)|\leqslant N,
\] 
for $n=1,2,\dotsc$. Since $\mathscr{F}$ is finite, we can apply Lemma~\ref{lemA} (reverse form) to find a subsequence $(v_n)$ of $(u_n)$ with $v_n=w_nf_nf_{n-1}\dotsc f_1$, where $f_i\in\mathscr{F}$ and $w_n\in \mathscr{F}^+$. Let $F_n=f_n\circ f_{n-1}\circ \dotsb\circ  f_1$. Then $F(v_n)=F(w_n)\circ F_n$, so
\[
|F_n(a)| \leqslant |F(w_n)(F_n(a))|=|F(v_n)(a)|\leqslant N.
\]
Additionally, $|F_{n+1}(a)| = |f_{n+1}(F_n(a))| \geqslant |F_n(a)|$, so the sequence $|F_1(a)|,|F_2(a)|,\dotsc$ is eventually constant. Then, because $\mathscr{A}$ is finite, we can find positive integers $r$ and $s$ (with $r>s$) for which $|F_r(a)|=|F_s(a)|$ and $\pi_1(F_r(a))=\pi_1(F_s(a))$. We can now apply Lemma~\ref{lemB} with $f=f_r\circ f_{r-1}\circ \dotsb \circ f_{s+1}$ and $u=F_s(a)$ to see that $f$ fixes the letter $\pi_1(u)$, as required.\end{proof}

\section{Forward limit sets}\label{secD}

In this section we explore some of the basic properties of forward limit sets. Let $S$ be a semigroup generated by a finite set of substitutions $\mathscr{F}$ of a finite alphabet $\mathscr{A}$. We recall that the forward limit set of a subset $A$ of $\mathscr{A}$ for $S$ is $\Lambda(A)=\overline{S(A)}\setminus S(A)$. Since $\mathscr{A}^+$ has the discrete topology, the set $\Lambda(A)$ is contained wholly within $\mathscr{A}^\mathbb{N}$.  Evidently it is a closed subset of  $\mathscr{A}^\mathbb{N}$ (it is the boundary of $S(A)$ in $\widetilde{\mathscr{A}}$). We will prove that it is $S$-invariant.

\begin{lemma}\label{lemS}
The forward limit set $\Lambda(A)$ of a subset $A$ of $\mathscr{A}$ for a substitution semigroup $S$ is $S$-invariant.  
\end{lemma}
\begin{proof}
Let $x\in \Lambda(A)$ and $g\in S$. Then there exist sequences $(F_n)$ in $S$ and $(a_n)$ in $A$ with $F_n(a_n)\to x$. Since $\mathscr{A}$ is finite we can assume by passing to a subsequence that in fact all the letters $a_n$ are equal, to some letter $a$. Hence $F_n(a)\to x$. Observe that $(g\circ F_n)$ is also a sequence in $S$ and $g\circ F_n(a)\to g(x)$ by continuity of $g$. Hence $g(x)\in \Lambda(A)$. Thus $\Lambda(A)$ is $S$-invariant, as required. 
\end{proof}

We record another elementary lemma, used often. Here and henceforth we write $\Lambda(a)$ for $\Lambda(\{a\})$, where $a$ is a letter of $\mathscr{A}$.

\begin{lemma}
Let $S$ be a substitution semigroup. Then 
\[
\Lambda(A) = \displaystyle\bigcup_{a\in A}\Lambda(a),
\]
for any subset $A$ of $\mathscr{A}$.
\end{lemma}
\begin{proof}
We have $S(A)=\bigcup_{a\in A} S(a)$, by definition. Hence $\overline{S(A)}=\bigcup_{a\in A} \overline{S(a)}$ because $\mathscr{A}$ is finite. Note also that $\overline{S(a)}\setminus S(a)=\overline{S(a)}\setminus S(A)$, for $a\in A$, because $S(A)\subseteq \mathscr{A}^+$. Hence
\[
\bigcup_{a\in A} \Lambda(a)=\bigcup_{a\in A} (\overline{S(a)}\setminus S(a))=\bigcup_{a\in A} (\overline{S(a)}\setminus S(A))=\overline{S(A)}\setminus S(A)=\Lambda(A),
\]
as required. 
\end{proof}

Thus we can describe all forward limit sets for $S$ in terms of the forward limit sets $\Lambda(a)$, where $a\in\mathscr{A}$. We will explore how to use the first-letter graph $G_{\mathscr{F}}$ of $\mathscr{F}$ to understand how the forward limit sets of individual letters are related. The next two lemmas will assist us with this task.

\begin{lemma}\label{lemD}
Let $g$ be a substitution of an alphabet $\mathscr{A}$. Let $x,y\in\widetilde{\mathscr{A}}$, and suppose that $\pi_1(x)=\pi_1(y)$. Then 
\[
d(g(x),g(y))\leqslant \frac{1}{2^{|g(a)|}},
\]
where $a=\pi_1(x)$.
\end{lemma}
\begin{proof}
Let $x=a_1a_2\dotsc$ and $y=b_1b_2\dotsc$, for $a_i,b_j\in\mathscr{A}$, where $a_1=b_1=a$. Then $g(x)=g(a_1)g(a_2)\dotsc$ and $g(y)=g(b_1)g(b_2)\dotsc$. Consequently, the first $|g(a)|$ letters of $g(x)$ and $g(y)$ coincide, so the inequality follows.
\end{proof}

\begin{lemma}\label{lemK}
 Let $(x_n)$ and $(y_n)$ be sequences in $\widetilde{\mathscr{A}}$ with $\pi_1(x_n)=\pi_1(y_n)$, for $n=1,2,\dotsc$, and let $x\in\mathscr{A}^\mathbb{N}$. Suppose that $(F_n)$ is a sequence of substitutions of $\mathscr{A}$ for which $F_n(x_n)\to x$ and $|F_n(a_n)|\to\infty$, where $a_n=\pi_1(x_n)$. Then $F_n(y_n)\to x$.
\end{lemma}
\begin{proof}
By Lemma~\ref{lemD}, we have
\[
d(F_n(x_n),F_n(y_n))\leqslant d(F_n(x_n),F_n(a_n))+d(F_n(a_n),F_n(y_n))\leqslant \frac{2}{2^{|F_n(a_n)|}}.
\]
Since $|F_n(a_n)|\to\infty$, we see that $F_n(y_n)\to x$, as required.
\end{proof}

Now we return to relating forward limit sets using the first-letter graph.

\begin{lemma}\label{lemE}
Let $S$ be a substitution semigroup with finite generating set $\mathscr{F}$. Suppose that there is a directed walk in $G_{\mathscr{F}}$ from one letter $a$ to another letter $b$. Then $\Lambda(b)\subseteq\Lambda(a)$.
\end{lemma}
\begin{proof}
That there is a directed walk in $G_{\mathscr{F}}$ from $a$ to $b$ implies that there is an element $g$ of $S$ with $g[a]=b$. We define $w=g(a)$, so $\pi_1(w)=b$. Suppose now that $x\in\Lambda(b)$. Then there is a sequence $(F_n)$ in $S$ with $F_n(b)\to x$.  Since $|F_n(b)|\to\infty$ we can apply Lemma~\ref{lemK} to see that $F_n(w)\to x$. Now, $F_n(w)=F_n\circ g(a)$, and $(F_n\circ g)$ is itself a sequence in $S$, so we see that $x\in \Lambda(a)$, as required. 
\end{proof}

An immediate corollary of Lemma~\ref{lemE} is that the forward limit sets of letters in the same strongly connected component of $G_{\mathscr{F}}$ coincide.

\begin{corollary}\label{corC}
Let $S$ be a substitution semigroup with finite generating set $\mathscr{F}$. Suppose that letters $a$ and $b$ lie in the same strongly connected component of  $G_{\mathscr{F}}$. Then $\Lambda(a)=\Lambda(b)$.
\end{corollary}

We finish this section with one final theorem on forward limit sets. In this theorem we use the notation $\pi_1(X)$ for the set $\{\pi_1(x): x\in X\}$. In the proof we use the map $F\colon \mathscr{F}^+\longrightarrow S$ defined in Section~\ref{secC} that sends $f_1f_2\dotsc f_n$ to $f_1\circ f_2\circ \dots \circ f_n$.

\begin{theorem}
Let $S$ be a fixed-letter-free substitution semigroup with finite generating set $\mathscr{F}$, and let $X$ be a forward limit set of some subset of $\mathscr{A}$ for $S$. Then $X=\Lambda(A)$, where $A=\pi_1(X)$.
\end{theorem}
\begin{proof}
First we show that $\Lambda(A)\subseteq X$. To this end, choose $x\in \Lambda(A)$. Then there exists $a\in A$ and a sequence $(F_n)$ in $S$ with $F_n(a)\to x$. Since $A=\pi_1(X)$, we can find $y\in X$ with $\pi_1(y)=a$. Then $F_n(y)\to x$ by Lemma~\ref{lemK}. Observe that $F_n(y)\in X$, by $S$-invariance of $X$, and then $x\in X$, by closure of $X$. Hence $\Lambda(A) \subseteq X$.

Next we show that $X\subseteq \Lambda(A)$. Observe that $X=\Lambda(B)$, for some subset $B$ of $\mathscr{A}$. Choose $x\in X$. Then there exists $b\in B$ and a sequence of words $(w_n)$ over $\mathscr{F}$ for which $F(w_n)(b)\to x$. Note that $|w_n|\to\infty$. Choose further sequences $(u_n)$ and $(v_n)$ in $\mathscr{F}^+$ with $w_n=u_nv_n$ and $|u_n|\to\infty$ and $|v_n|\to\infty$. Since $\mathscr{A}$ is finite we can assume by passing to a subsequence that in fact all the letters $F(v_n)[b]$, for $n=1,2,\dotsc$, are equal to some letter $c$ in $\mathscr{A}$. 

Now, because $(\widetilde{\mathscr{A}},d)$ is compact, there is a subsequence of $(F(v_n)(b))$ that converges to some limit $y\in\widetilde{\mathscr{A}}$. Since $|v_n|\to\infty$ we see that $|F(v_n)(b)|\to\infty$ by Theorem~\ref{thmG}. Hence $y\in\mathscr{A}^{\mathbb{N}}$, so $y\in\Lambda(B)$. Since $F(v_n)[b]=c$ for each $n$, it follows that $\pi_1(y)=c$. So $c\in A$.

We have $F(u_n)(F(v_n)(b))\to x$ and $\pi_1(F(v_n)(b))=c$ for each $n$. Observe that $|F(u_n)(c)|\to\infty$, by Theorem~\ref{thmG}, because $|u_n|\to\infty$. Hence we can apply Lemma~\ref{lemK} with $x_n=F(v_n)(b)$, $y_n=c$, and $F_n=F(u_n)$ to see that $F(u_n)(c)\to x$. Hence $x\in \Lambda(A)$, so $X\subseteq \Lambda(A)$. Therefore $X= \Lambda(A)$, as required.
\end{proof}

\section{Forward limit sets and s-adic limits}\label{secE}

In this section we prove Theorem~\ref{thmA}. That theorem says that for a fixed-letter-free substitution semigroup $S$ generated by a finite set of substitutions $\mathscr{F}$, the forward limit set of $S$ is equal to the set of all s-adic limits of $\mathscr{F}$. The proof uses the map $F$ from $\mathscr{F}^+$ to $S$ defined in Section~\ref{secC}.

\begin{proof}[Proof of Theorem~\ref{thmA}]
Suppose that $x$ is an s-adic limit of $\mathscr{F}$. Then there are sequences $(f_n)$ in $\mathscr{F}$  and $(a_n)$ in $\mathscr{A}$ for which $F_n(a_n)\to x$ as $n\to\infty$, where $F_n=f_1\circ f_2\circ \dots \circ f_n$. Since $\mathscr{A}$ is finite, we can find an increasing sequence of positive integers $(n_k)$ and a letter $a\in\mathscr{A}$ with $a_{n_k}=a$, for $k=1,2,\dotsc$. Hence $F_{n_k}(a)\to x$ as $n_k\to\infty$. Consequently, $x\in\Lambda(a)$, so $x\in\Lambda$.

Conversely, suppose that $x\in \Lambda$. Then $x\in\Lambda(a)$, for some letter $a\in \mathscr{A}$. Hence there is a sequence of words $(u_n)$ in $\mathscr{F}^+$ with $|u_n|\to\infty$ and $F(u_n)(a)\to x$. By Lemma~\ref{lemA}, there is a subsequence $(v_n)$ of $(u_n)$ of the form $v_n=f_1f_2\dotsc f_n w_n$, for $n=1,2,\dotsc$, where $f_i\in\mathscr{F}$ and $w_n\in \mathscr{F}^+$. Let $F_n=f_1\circ f_2\circ \dotsb \circ f_n$,  $x_n=F(w_n)(a)$, and $a_n=\pi_1(x_n)$, for $n=1,2,\dotsc$. Then $F_n(x_n)=F(v_n)(a)$, so $F_n(x_n)\to x$. Since $|F_n(a_n)|\to \infty$, by Theorem~\ref{thmG}, we can apply Lemma~\ref{lemK} with $y_n=a_n$ to see that $F_n(a_n)\to x$. Therefore $x$ is an s-adic limit of $\mathscr{F}$, as required.
\end{proof}

Theorem~\ref{thmA} tells us that for each $x\in\Lambda$ we can find sequences $(f_n)$ in $\mathscr{F}$ and $(a_n)$ in $\mathscr{A}$ with $f_1\circ f_2\circ \dotsb \circ f_n(a_n)\to x$ as $n\to\infty$. The next theorem says that we can obtain a similar conclusion but with $f_n\in S$ (rather than $f_n\in\mathscr{F}$) and all letters $a_n$ equal.

\begin{theorem}\label{thmH}
Let $S$ be a fixed-letter-free substitution semigroup with finite generating set $\mathscr{F}$, and let $x\in\Lambda$. Then there exists a letter $a\in\mathscr{A}$ and a sequence $(g_n)$ in $S$ for which
\[
g_1\circ g_2\circ \dots \circ g_n(a)\to x \quad\text{as $n\to\infty$}.
\]
\end{theorem}
\begin{proof}
By Theorem~\ref{thmA}, $x$ is an s-adic limit  of $\mathscr{F}$. Hence there exist sequences $(f_n)$ in $\mathscr{F}$ and $(a_n)$ in $\mathscr{A}$ for which 
\(F_n(a_n)\to x\), where $F_n=f_1\circ f_2\circ \dots \circ f_n$. Since $\mathscr{A}$ is finite we can find a letter $a$ in $\mathscr{A}$ and an increasing sequence of positive integers $(n_k)$ with $a_{n_k}=a$, for $k=1,2,\dotsc$. Let $n_0=0$ and $g_k=f_{n_{k-1}+1}\circ f_{n_{k-1}+2}\circ \dots \circ f_{n_k}$, for $k=1,2,\dotsc$, and define $G_k=g_1\circ g_2\circ \dotsb \circ g_k$. Then $G_k=F_{n_k}$. Hence
\[
G_k(a) = F_{n_k}(a)=F_{n_k}(a_{n_k})\to x
\]
as $k\to\infty$, as required.
\end{proof}

\section{Invariant sets}\label{secF}

In this section we prove Theorem~\ref{thmB} and related results. A first step towards these goals is the following lemma. Recall that a subset $X$ of $\mathscr{A}^\mathbb{N}$ is strongly $S$-invariant if $S(X)=X$.

\begin{lemma}\label{lemJ}
Let $S$ be a fixed-letter-free substitution semigroup with finite generating set $\mathscr{F}$. The forward limit set $\Lambda(a)$ of a letter $a\in\mathscr{A}$ is strongly $S$-invariant.
\end{lemma}
\begin{proof}
We already know that $\Lambda(a)$ is $S$-invariant, by Lemma~\ref{lemS}, so it remains to show that if $y\in \Lambda(a)$, then there exists $g\in S$ and $x\in\Lambda(a)$ with $g(x)=y$. To prove this, observe that since $y\in\Lambda(a)$ there is a sequence $(w_n)$ in $\mathscr{F}^+$ with $F(w_n)(a)\to y$. Note that $|w_n|\to\infty$. By restricting to a subsequence of $(w_n)$, we can assume that in fact $w_n=gv_n$, for $n=1,2,\dotsc$, where $g\in\mathscr{F}$ and $v_n\in\mathscr{F}^+$. Now, $(F(v_n)(a))$ is a sequence in the compact metric space $(\widetilde{\mathscr{A}},d)$, so there is an increasing sequence of positive integers $(n_k)$ for which $F(v_{n_k})(a)\to x$, for some point $x\in \widetilde{\mathscr{A}}$. In fact, $x\in\mathscr{A}^{\mathbb{N}}$, because $|F(v_{n_k})(a)|\to\infty$ by Theorem~\ref{thmG}. Hence, by continuity of $g$, 
\[
F(w_{n_k})(a)=g\circ F(v_{n_k})(a)\to g(x),
\]
so $g(x)=y$, as required.
\end{proof}

It follows immediately from Lemma~\ref{lemJ} that the forward limit set of any subset of $\mathscr{A}$ for $S$ is strongly $S$-invariant.

\begin{lemma}\label{lemL}
Let $S$ be a fixed-letter-free substitution semigroup with finite generating set $\mathscr{F}$, and let $X$ be a subset of $\mathscr{A}^{\mathbb{N}}$. Suppose that for some $x\in X$ there are sequences $(f_n)$ in $S$ and $(x_n)$ in $X$ with $x_{n-1}=f_n(x_n)$, for $n=1,2,\dotsc$, where $x_0=x$. Then $x\in\Lambda(A)$, where $A=\pi_1(X)$.
\end{lemma}
\begin{proof}
Choose sequences $(f_n)$ and $(x_n)$ as specified in the lemma.  Let $F_n=f_1\circ f_2\circ \dots\circ f_n$. Then we have $F_n(x_n)=x$. Define $a_n=\pi_1(x_n)$, and note that $a_n\in \pi_1(X)$. Since $|F_n(a_n)|\to \infty$, by Theorem~\ref{thmG}, we can apply Lemma~\ref{lemK} with $y_n=a_n$ to see that $F_n(a_n)\to x$. Hence $x\in \Lambda(A)$, as required.
\end{proof}

Using Lemma~\ref{lemL} we obtain another characterisation of $\Lambda$.

\begin{theorem}\label{thmI}
Let $S$ be a fixed-letter-free substitution semigroup with finite generating set $\mathscr{F}$, and let $x\in\mathscr{A}^{\mathbb{N}}$. Then $x\in\Lambda$ if and only if there are sequences $(f_n)$ in $S$ and $(x_n)$ in $\mathscr{A}^{\mathbb{N}}$ with $x_{n-1}=f_n(x_n)$, for $n=1,2,\dotsc$, where $x_0=x$.
\end{theorem}
\begin{proof}
Suppose first that $x\in\Lambda$. Since $\Lambda$ is strongly $S$-invariant, we can find a substitution $f_1$ in $S$ and an infinite word $x_1$ in $\mathscr{A}^{\mathbb{N}}$ with $f_1(x_1)=x$. By repeating this argument (and applying the axiom of dependent choices) we obtain the required sequences $(f_n)$ and $(x_n)$.

Conversely, suppose there are sequences $(f_n)$ in $S$ and $(x_n)$ in $\mathscr{A}^{\mathbb{N}}$ with $x_{n-1}=f_n(x_n)$, for $n=1,2,\dotsc$, where $x_0=x$. Applying Lemma~\ref{lemL} with $X=\mathscr{A}^\mathbb{N}$, we see that $x\in \Lambda(A)$, where $A=\pi_1(X)=\mathscr{A}$. Consequently, $x\in\Lambda$, as required.
\end{proof}

The next lemma features in the proof of Theorem~\ref{thmB}.

\begin{lemma}\label{lemM}
Let $S$ be a substitution semigroup with finite generating set $\mathscr{F}$, and let $X$ be a closed and $S$-invariant subset of $\mathscr{A}^{\mathbb{N}}$. Then $ \Lambda(A)\subseteq X$, where $A=\pi_1(X)$.
\end{lemma}
\begin{proof}
Let $y\in \Lambda(A)$. There exists $a\in A$ and a sequence $(F_n)$ in $S$ with $F_n(a)\to y$. Choose $x\in X$ with $\pi_1(x)=a$. Then $F_n(x)\to y$ by Lemma~\ref{lemK}. Observe that $F_n(x)\in X$, by $S$-invariance of $X$, and then $y\in X$, by closure of $X$. Hence $\Lambda(A)\subseteq X$, as required.
\end{proof}

We can now prove Theorem~\ref{thmB}. In fact, we prove the following more precise version of that theorem.

\begin{theorem}\label{thmJ}
Let $S$ be a fixed-letter-free substitution semigroup with finite generating set $\mathscr{F}$, and let $X\subseteq\mathscr{A}^{\mathbb{N}}$. Then $X$ is closed and strongly $S$-invariant if and only if $X=\Lambda(A)$, where $A=\pi_1(X)$.
\end{theorem}
\begin{proof}
Suppose that $X=\Lambda(A)$, where $A=\pi_1(X)$. Then $X$ is closed, and it is strongly $S$-invariant, by Lemma~\ref{lemJ}.

Conversely, suppose that $X$ is closed and strongly $S$-invariant. Let $x\in X$. By the strongly $S$-invariant property, there are sequences $(f_n)$ in $S$ and $(x_n)$ in $X$ with $f_n(x_n)=x_{n-1}$, for $n=1,2,\dotsc$, where $x_0=x$. Applying Lemma~\ref{lemL}, we find that $x\in\Lambda(A)$, where $A=\pi_1(X)$. Thus $X\subseteq \Lambda(A)$. Next, $\Lambda(A)\subseteq X$, by Lemma~\ref{lemM}. Hence $X=\Lambda(A)$, as required.
\end{proof}

We finish this section with a result (Theorem~\ref{thmC}) characterising closed and $S$-invariant sets (rather than closed and strongly $S$-invariant sets) in terms of forward limit sets. A \emph{minimal} closed and $S$-invariant subset of $\mathscr{A}^\mathbb{N}$ is a closed and $S$-invariant subset of $\mathscr{A}^\mathbb{N}$ that contains no other such sets besides itself (and the empty set). The next lemma facilitates proving Theorem~\ref{thmC}.

\begin{lemma}\label{lemR}
Let $S$ be a substitution semigroup of the alphabet $\mathscr{A}$, and let $x\in\Lambda(A)$, for some subset $A$ of $\mathscr{A}$. Then there exists $a\in A$ and $g\in S$ with $g[a]=\pi_1(x)$. 
\end{lemma}
\begin{proof}
Since $x\in\Lambda(A)$, there exists $a\in A$ and a sequence $(F_n)$ in $S$ with $F_n(a)\to x$. Thus there exists a positive integer $m$ with $\pi_1(F_n(a))=\pi_1(x)$, for $n\geqslant m$. Hence we can choose $g=F_m$, to give 
$g[a]=F_m[a]=\pi_1(F_m(a))=\pi_1(x)$, as required.
\end{proof}

Here is the promised result on closed and $S$-invariant sets.

\begin{theorem}\label{thmC}
Let $S$ be a fixed-letter-free substitution semigroup with finite generating set $\mathscr{F}$. A subset of $\mathscr{A}^{\mathbb{N}}$ is a minimal closed and $S$-invariant subset if and only if it is the forward limit set of a terminal component of $G_{\mathscr{F}}$. 
\end{theorem}
\begin{proof}
Let $T$ be a terminal component of $G_{\mathscr{F}}$. Then $\Lambda(T)$ is closed and $S$-invariant. Let $X$ be a closed and $S$-invariant subset of $\mathscr{A}^{\mathbb{N}}$ with $X\subseteq \Lambda(T)$. Observe that $\Lambda(A)\subseteq X$, where $A=\pi_1(X)$, by Lemma~\ref{lemM}. We will prove that $\Lambda(A)=\Lambda(T)$, from which it follows that $\Lambda(T)$ is minimal.

To do this, choose $a\in A$. Then $a=\pi_1(x)$, for some $x\in X$, so $x\in \Lambda(T)$, because $X\subseteq \Lambda(T)$. By Lemma~\ref{lemR}, there exists $b\in T$ and $g\in S$ with $g[b]=a$. But $T$ is a terminal component, so there exists $h\in S$ with $h[a]=b$. Hence $\Lambda(b)\subseteq \Lambda(a)$, by Lemma~\ref{lemE}. Observe that $\Lambda(b)=\Lambda(T)$, by Corollary~\ref{corC}, and of course $\Lambda(a)\subseteq \Lambda(A)$, so we see that $\Lambda(T)\subseteq \Lambda(A)$. Hence $\Lambda(A)=\Lambda(T)$, as required.

Conversely, let $X$ be a minimal closed and $S$-invariant subset of $\mathscr{A}^\mathbb{N}$. Then $\Lambda(A)\subseteq X$, where $A=\pi_1(X)$, by Lemma~\ref{lemM}. Choose $a\in A$. Then $\Lambda(a)\subseteq \Lambda(A)$. Since $\Lambda(a)$ is closed and $S$-invariant, we must have $X=\Lambda(a)$, by minimality of $X$.

Now choose $x\in X$ with $\pi_1(x)=a$. There exists $f\in S$ with $f[a]=a$, by Lemma~\ref{lemR}. Hence the vertex $a$ belongs to some strongly connected component $B$ of $G_{\mathscr{F}}$.

Let $g\in S$ and define $b=g[a]$. To show that $B$ is a terminal component of $G_{\mathscr{F}}$ we must prove that $b\in B$. Observe that $\Lambda(b)\subseteq \Lambda(a)$, by Lemma~\ref{lemE}, so $X=\Lambda(a)=\Lambda(b)$ by minimality of $X$. Then, by Lemma~\ref{lemR}, there exists $h\in S$ with $h[b]=a$. Therefore $b\in B$, as required.
\end{proof}

\section{Fixed points of substitutions semigroups}\label{secG}

In this section we prove Theorem~\ref{thmD}. This theorem says that the forward limit set of a finitely-generated fixed-letter-free substitution semigroup $S$ is equal to $\overline{S(X)}$, where $X$ is the set of fixed points of $S$. Three preliminary results are required before we prove the theorem itself. The first one is well known and straightforward, so we omit the proof.

\begin{lemma}\label{lemI}
Let $f$ be a substitution that has no fixed letters and that satisfies $f[a]=a$, for some letter $a$. Then $f$ has a unique fixed point $x$ in $\mathscr{A}^{\mathbb{N}}$ with first letter $a$. Furthermore, $f^n(a)\to x$ as $n\to\infty$.
\end{lemma}

For the next lemma, recall that $\fix(S)$ denotes the set of fixed points of $S$.

\begin{lemma}\label{lemG}
Let $S$ be a fixed-letter-free substitution semigroup with finite generating set $\mathscr{F}$, and let $a\in\mathscr{A}$. Let $(f_n)$ be a sequence in $S$ and let $F_n=f_1\circ f_2\circ\dots\circ f_n$. Suppose that some subsequence of $(F_n(a))$ converges to an infinite word $x$ in $\mathscr{A}^{\mathbb{N}}$, where $\pi_1(x)=a$. Then $x\in \overline{\fix(S)}$. 
\end{lemma}
\begin{proof}
Let us choose a sequence of positive integers $(n_k)$ for which $F_{n_k}(a)\to x$. Since $\pi_1(x)=a$ it follows that $\pi_1(F_{n_k}(x))=a$ for sufficiently large values of $k$; in fact, after relabelling the sequence we  can assume that  $\pi_1(F_{n_k}(x))=a$, for $k=1,2,\dotsc$. 

By Lemma~\ref{lemI}, $F_{n_k}$ has a unique fixed point $x_k$ in $\mathscr{A}^{\mathbb{N}}$ with $\pi_1(x_k)=a$. And by Lemma~\ref{lemK}, $F_{n_k}(x_k)\to x$. Since $F_{n_k}(x_k)=x_k$, we deduce that $x\in \overline{\fix(S)}$, as required.
\end{proof}

\begin{lemma}\label{lemF}
Let $(w_n)$ be a sequence in $\mathscr{A}^+$ with $|w_1|<|w_2|<\dotsb$. Then there exists a pair $(r,a)\in\mathbb{N}\times\mathscr{A}$ and two increasing sequences of positive integers $(n_k)$ and $(l_k)$ with $r+l_k\leqslant |w_{n_k}|$ and 
\[
\pi_r(w_{n_k})=\pi_{r+l_k}(w_{n_k})=a,
\]
for $k=1,2,\dotsc$.
\end{lemma}
\begin{proof}
We prove the lemma by induction on the order $m$ of $\mathscr{A}$. Suppose that $m=1$. Then we choose $a$ to be the unique element of $\mathscr{A}$, and we choose $r=1$, $n_k=k+1$, and $l_k=|w_{k+1}|-1$, for $k=1,2,\dotsc$. This pair and these sequences satisfy the required properties.

Suppose now that the lemma is true for alphabets of order $m$, where $m$ is a positive integer. Let $\mathscr{A}$ be an alphabet of order $m+1$. Since $\mathscr{A}$ is finite, there is an element $e$ of $\mathscr{A}$ and an infinite subset $E$ of $\mathbb{N}$ with $\pi_1(w_n)=e$, for all $n\in E$.

Now, there may be increasing sequences of positive integers  $(n_k)$ and $(l_k)$ with $(n_k)$ in $E$ and with $1+l_k\leqslant |w_{n_k}|$ and $\pi_1(w_{n_k})=\pi_{1+l_k}(w_{n_k})=e$, for $k=1,2,\dotsc$. If that is so, then the pair $(1,e)$ and the sequences $(n_k)$ and $(l_k)$ satisfy the required properties.

If there are no such sequences, then there is an integer $s$ for which $\pi_j(w_n)\neq e$, for $j> s$ and $n\in E$. Let $E'$ denote the infinite subset of $E$ of those integers $n$ in $E$ with $|w_n|>s$. Let $p_1,p_2,\dotsc$ be the elements of $E'$, in increasing order, and define $v_n$ to be the word obtained from $w_{p_n}$ by removing the first $s$ letters from $w_{p_n}$. Then $(v_n)$ is a sequence in $(\mathscr{A}-\{e\})^+$
with $|v_1|<|v_2|<\dotsb$. Since $\mathscr{A}-\{e\}$ has order $m$ we can apply the inductive hypothesis to obtain a pair  $(r,a)\in\mathbb{N}\times(\mathscr{A}-\{e\})$ and two increasing sequences of positive integers $(n_k)$ and $(l_k)$ with $r+l_k\leqslant |v_{n_k}|$ and $\pi_r(v_{n_k})=\pi_{r+l_k}(v_{n_k})=a$, for $k=1,2,\dotsc$. Hence $(r+s, a)\in\mathbb{N}\times\mathscr{A}$ and $(r+s)+l_k\leqslant |w_{p_{n_k}}|$ and $\pi_{r+s}(w_{p_{n_k}})=\pi_{(r+s)+l_k}(w_{p_{n_k}})=a$, for $k=1,2,\dotsc$. Therefore the pair $(r+s,a)$ and the sequences $(p_{n_k})$ and $(l_k)$ satisfy the properties required by the lemma. This completes the proof by induction.
\end{proof}

We can now prove Theorem~\ref{thmD}.

\begin{proof}[Proof of Theorem~\ref{thmD}]
Let $X$ denote the collection of fixed points of $S$; that is, $X=\fix(S)$. 

Since $S$ has no fixed letters, it also has no fixed finite words, by Corollary~\ref{corB}. Suppose now that $x\in X$. Then $x$ is an infinite word fixed by some subtitution $f\in S$.  Lemma~\ref{lemI} tells us that $f^n(a)\to x$, where $a=\pi_1(x)$. Hence $x\in \Lambda$. It follows that $X\subseteq \Lambda$, and since $\Lambda$ is closed and $S$-invariant, we see that $\overline{S(X)}$ is contained in $\Lambda$.

Conversely, take $x\in \Lambda$. Then by Theorem~\ref{thmH} we can find $b\in\mathscr{A}$ and substitutions $g_1,g_2,\dotsc$ in $S$ for which $G_{n}(b)\to x$, where $G_n=g_1\circ g_2\circ \dots \circ g_n$. Let us define $G_{m,n}=g_m\circ g_{m+1}\circ \dotsb \circ g_n$, where $1\leqslant m\leqslant n$. We also define the $n$-letter word
\[
w_n=G_{n}[b]G_{2,n}[b]\dotsc G_{n,n}[b],
\]
for $n=1,2,\dotsc$. By Lemma~\ref{lemF}, there exists a pair $(r,a)\in\mathbb{N}\times\mathscr{A}$ and two increasing sequences of positive integers $(n_k)$ and $(l_k)$ with $r+l_k\leqslant |w_{n_k}|$ and 
\[
G_{r,n_k}[b]=G_{r+l_k,n_k}[b]=a.
\]
We define $g=G_{r-1}$ (possibly $r=1$ in which case instead we define $g$ to be the identity substitution) and $h_n=g_{n+r-1}$ and $H_n=h_1\circ h_2\circ\dotsb \circ h_n$, for $n=1,2,\dotsc$. Then
\[
H_{l_k}[a]= H_{l_k}[G_{r+l_k,n_k}[b]]=G_{r,n_k}[b]=a,
\]
for $k=1,2,\dotsc$. 

Since $(\mathscr{A}^\mathbb{N},d)$ is compact, we can find a convergent subsequence $(H_{p_k}(a))$ of $(H_{l_k}(a))$. The limit $y$ of this sequence must have first letter $a$, because $H_{l_k}[a]=a$, for all $k$. By Lemma~\ref{lemG}, $y\in\overline{X}$. Hence  $g\circ H_{p_k}(a)\to g(y)$, by continuity of $g$, and $g(y)\in\overline{S(X)}$. 

Next, by applying Lemma~\ref{lemK} with $x_k=a$, $y_k=G_{r+p_k,n_k}(b)$, and $F_k=H_{p_k}$, we see that $H_{p_k}(G_{r+p_k,n_k}(b))\to y$. But $H_{p_k}\circ G_{r+p_k,n_k}= G_{r,n_k}$, so $G_{r,n_k}(b)\to y$. Therefore $G_{n_k}(b)=g\circ G_{r,n_k}(b)\to g(y)$. However, $G_n(b)\to x$, which implies that $x=g(y)$. Thus $x\in\overline{S(X)}$, as required.
\end{proof}

\section{Cardinality of forward limit sets}\label{secH}

In this section we prove Theorem~\ref{thmE}, which gives sufficient conditions for the forward limit set of a substitution semigroup to be uncountable. We recall that a word $a$ is said to be a prefix of another word $b$ if $b=ac$, where $c$ is either another word or the empty word.

\begin{lemma}\label{lemH}
Let $f$ be a substitution of $\mathscr{A}$ for which $f(a)$ is not a prefix of $f(b)$, for any distinct $a,b\in\mathscr{A}$. Then $f$ is injective on $\widetilde{\mathscr{A}}$.
\end{lemma}
\begin{proof}
Consider two infinite words $u=a_1a_2\dotsc$ and $v=b_1b_2\dotsc$ over $\mathscr{A}$, where $a_i$ and $b_j$ are letters of $\mathscr{A}$. Suppose that $f(u)=f(v)$. Then 
\[
f(a_1)f(a_2)\dotsc = f(b_1)f(b_2)\dotsc,
\]
so one of $f(a_1)$ and $f(b_1)$ must be a prefix of the other. Consequently, $a_1=b_1$. Now let $u'=a_2a_3\dotsc$ and $v'=b_2b_3\dotsc$. Then $f(u')=f(v')$, and we obtain $a_2=b_2$. This is the basis for an inductive argument to show that $a_n=b_n$, for each positive integer $n$, so $u=v$. Thus $f$ is injective on infinite words.

The same sort of argument can be used to show that $f$ is also injective on finite words. Hence $f$ is injective on $\widetilde{\mathscr{A}}$, as required.
\end{proof}

We can now prove Theorem~\ref{thmE}.

\begin{proof}[Proof of Theorem~\ref{thmE}]
We begin by observing that each substitution from $S$ is injective in its action on $\mathscr{A}^\mathbb{N}$, because the generators from $\mathscr{F}$ are injective, by Lemma~\ref{lemH}. 

Next, let $a$ be a letter from $\mathscr{A}$ for which there are at least two words in $\fix(S)$ with first letter $a$. We denote by $\Sigma(a)$ the set of those infinite words $x$ with first letter $a$ for which there is a sequence of substitutions $(g_n)$ in $S$ with $G_n(a)\to x$, where $G_n=g_1\circ g_2\circ \dotsb \circ g_n$. This set $\Sigma(a)$ contains all fixed points of $S$ with first letter $a$ (by Lemma~\ref{lemI}), so it has size at least two.
Observe that $\Sigma(a)$ is invariant under those elements $h$ of $S$ that satisfy $h[a]=a$.

Choose any infinite word $x$ in $\Sigma(a)$. Let $y$ be another infinite word in $\Sigma(a)$ and let $\varepsilon$ be any positive number that satisfies $\varepsilon<d(x,y)$. Let $(g_n)$ be a sequence in $S$ with $G_n(a)\to x$. By Lemma~\ref{lemG}, we have $x\in\overline{\fix(S)}$. Hence there is an infinite word $z$  that is fixed by some substitution $h\in S$ and that satisfies $d(x,z)<\varepsilon$. Since $\varepsilon<1$
it follows that $\pi_1(z)=a$ (so $z\in\Sigma(a)$) and $z\neq y$. 

Notice that $h^n(a)\to z$, by Lemma~\ref{lemI}, so $h^n(y)\to z$, by Lemma~\ref{lemK}. Each iterate $h^n(y)$ lies in $\Sigma(a)$, by invariance of $\Sigma(a)$ under $h$. Also, none of the iterates equal $z$, because $h$ is injective and $h(z)=z$. Consequently, $\Sigma(a)$ accumulates at $z$. Since $\varepsilon$ can be chosen to be arbitrarily small, we see that $\Sigma(a)$ also accumulates at $x$. It follows that $\overline{\Sigma(a)}$ is a perfect set.
Since $(\mathscr{A}^\mathbb{N},d)$ is complete, and the forward limit set $\Lambda$ of $S$ contains $\overline{\Sigma(a)}$, we deduce that $\Lambda$ is uncountable, as required.
\end{proof}

\section{Hausdorff dimension of forward limit sets}\label{secI}

In this section we prove Theorem~\ref{thmF} on the logarithmic Hausdorff dimension of forward limit sets. We begin with some background material on Hausdorff dimension, which can be found in \cite{Fa2003,HaKe1976,Ro1998}.  

We work in the metric space $(\mathscr{A}^\mathbb{N},d)$. For $\varepsilon>0$, an \emph{$\varepsilon$-cover} of a subset $X$ of $\mathscr{A}^\mathbb{N}$ is a sequence $(U_n)$ of subsets of $\mathscr{A}^\mathbb{N}$ for which $X\subseteq \bigcup_n U_n$, and such that, for each $n$, the diameter $\diam U_n$ of $U_n$ does not exceed $\varepsilon$. A \emph{dimension function} is a function $\phi\colon [0,+\infty) \longrightarrow [0,+\infty)$ that is increasing, continuous, and satisfies $\phi(0)=0$. Let 
\[
H^\phi_\varepsilon(X) = \inf \left\{ \sum_{n=1}^\infty {\phi(\diam U_n})\, :\,\text{$(U_n)$ is an $\varepsilon$-cover of $X$} \right\}
\]
and 
\[
H^\phi(X) = \lim_{\varepsilon\to 0}H^\phi_\varepsilon(X).
\]
The function $H^\phi$ is an outer measure on $\mathscr{A}^\mathbb{N}$  known as \emph{Hausdorff measure} with respect to~$\phi$.

Given two dimension functions $\phi$ and $\psi$ such that $\phi(x)/\psi(x)\to 0$ as $x\to 0$ it is straightforward to show that if $H^\psi(X)<+\infty$ then $H^\phi(X)=0$. The dimension functions most widely used are the collection $\phi_t(x)=x^t$, for $t>0$. For $t_1<t_2$ we have $\phi_{t_2}(x)/\phi_{t_1}(x) \to 0$ as $x\to 0$, so there is a unique value $d$  in $[0,+\infty]$ such that $H^{\phi_t}(X)=+\infty$, for $t<d$, and $H^{\phi_t}(X)=0$, for $t>d$. This value $d$ is the \emph{Hausdorff dimension of $X$}. 

An alternative collection of dimension functions is 
\[
\psi_t(x) = 
\begin{cases}
1/\mleft(\log \tfrac{1}{x}\mright)^{\!t} & \textnormal{if $0<x\leqslant 1/e$},\\
ex & \textnormal{otherwise},
\end{cases}
\]
 for $t>0$. It is only the value of $\psi_t$ near $0$ that matters; the extension to $[0,+\infty)$ is  chosen for convenience. For $t_1<t_2$ we have $\psi_{t_2}(x)/\psi_{t_1}(x) \to 0$ as $x\to 0$, so there is a unique value $d$ in $[0,+\infty]$ such that $H^{\psi_t}(X)=+\infty$, for $t<d$, and $H^{\psi_t}(X)=0$ for $t>d$. This value $d$ is the \emph{ logarithmic Hausdorff dimension of $X$}. Notice that, for any positive numbers $t_1$ and $t_2$, $\phi_{t_2}(x)/\psi_{t_1}(x) \to 0$ as $x\to 0$, which implies that if $X$ has finite logarithmic Hausdorff dimension, then it has Hausdorff dimension zero.

We recall from the introduction that the length of a finite set of substitutions $\mathscr{F}$ is the least length (number of letters) of $f(a)$, among all letters $a\in\mathscr{A}$ and substitutions $f\in\mathscr{F}$. 

The following lemma can be proved quickly by induction; we omit the details.

\begin{lemma}\label{lemN}
Let $S$ be a substitution semigroup generated by a finite set $\mathscr{F}$ of substitutions of length $r$. Then for any word $w$ over $\mathscr{F}$ of length $n$ and any letter $a$ in $\mathscr{A}$, $|F(w)(a)|$ has length at least $r^n$.
\end{lemma}

We can now prove the first part of Theorem~\ref{thmF}, which says that for a substitution semigroup generated by a finite set $\mathscr{F}$ of substitutions of length $r$ and size $s$, where $r>1$, the forward limit set of $S$ has logarithmic Hausdorff dimension at most $\log_rs$.

\begin{proof}[Proof of Theorem~\ref{thmF}: Part 1]
Let $\mathscr{F}^n$ denote the collection of words over $\mathscr{F}$ of length $n$. Given a word $w$ in $\mathscr{F}^n$ and $a\in\mathscr{A}$, we define the cylinder set
\[
X(w,a) = \{x\in\mathscr{A}^\mathbb{N}: \text{$x$ has prefix $F(w)(a)$}\},
\]
which is both closed and open. Let $\ell$ be the maximal possible length of $F(w)(a)$, among all words $w$ in $\mathscr{F}^n$ and letters $a\in\mathscr{A}$.

We will prove that the collection $\mathscr{X}_n=\{X(w,a): w\in\mathscr{F}^n,a\in\mathscr{A}\}$ covers $\Lambda$. To see this, choose $x\in\Lambda$. Then we can find a letter $b$ and a word $u$ over $\mathscr{F}$ for which $x$ and $F(u)(b)$ share a common prefix of length $\ell$. Let $w$ be the truncation of $u$ after $n$ letters (so $w\in\mathscr{F}^n$) and let $v$ be such that $u=wv$ (possibly $v$ is the empty word, in which case we define $F(v)(b)=b$). Let $a=\pi_1(F(v)(b))$. Then $F(w)(a)$ and $F(w)(F(v)(b))$ share a common prefix $F(w)(a)$, which has length less than or equal to $\ell$. But $F(w)(F(v)(b))=F(u)(b)$, so $x$ and $F(w)(a)$ share a common prefix $F(w)(a)$. Hence $x\in X(w,a)$.

Next, observe that $\diam(X(w,a))\leqslant 1/2^{r^n}$. Therefore, for any $\varepsilon>0$, the collection $\mathscr{X}_n$ is an $\varepsilon$-cover of $\Lambda$, for sufficiently large values of $n$. Using the dimension functions $\psi_t$ defined earlier in this section, we see that
\[
\psi_t(\diam(X(w,a)))\leqslant \frac{1}{(\log 2)^tr^{nt}}.
\]
Hence
\[
\sum_{X(w,a)\in\mathscr{X}_n}\psi_t(\diam(X(w,a))) \leqslant \frac{\lvert\mathscr{A}\rvert s^n}{(\log 2)^tr^{nt}}\leqslant \lvert\mathscr{A}\rvert\mleft(\frac{s}{r^{t}}\mright)^{\!n}.
\]
If $s/r^t<1$, or equivalently $t>\log_rs$, then $H^{\psi_t}(\Lambda)=0$. Therefore the logarithmic Hausdorff dimension of $\Lambda$ is less than or equal to $\log_rs$.
\end{proof}

We now turn to the other part of Theorem~\ref{thmF}, which says that, for any pair of positive integers $r$ and $s$ with $r>1$ and $s\leqslant |\mathscr{A}|^{r-1}$, there exists a substitution semigroup $S$ generated by a set of substitutions of length $r$ and size $s$ for which the logarithmic Hausdorff dimension of $\Lambda$ attains the bound $\log_rs$. In proving this we use the following lemma.

\begin{lemma}\label{lemO}
Let $\mathscr{F}$ be a set of $s$ substitutions for which $f[a]=a$, $|f(a)|=|g(a)|$, and $f(a)\neq g(a)$, for each $a\in\mathscr{A}$ and each distinct pair $f,g\in\mathscr{F}$. Then, for each positive integer $n$ and $a\in\mathscr{A}$, the $s^n$ words $f_1\circ f_2\circ \dots \circ f_n(a)$, where $f_i\in\mathscr{F}$, are distinct.
\end{lemma}
\begin{proof}
Let $a\in\mathscr{A}$. We assert that, for each positive integer $n$, the $s^n$ words $f_1\circ f_2\circ \dots \circ f_n(a)$, for $f_i\in\mathscr{F}$, are distinct.

This assertion is true for $n=1$ (one of the hypotheses of the lemma). Suppose that the assertion is true for $n=m-1$, where $m>1$. Now assume that $f_1\circ f_2\circ \dots \circ f_m(a)=g_1\circ g_2\circ \dots \circ g_m(a)$, for $f_i,g_i\in\mathscr{F}$. Observe that $f_2\circ f_3\circ \dots \circ f_m[a]=a$, so $f_1(a)$ is a prefix of $f_1\circ f_2\circ \dots \circ f_m(a)$ of length $|f_1(a)|$. Likewise $g_1(a)$ is a prefix of $g_1\circ g_2\circ \dots \circ g_m(a)$ of length $|g_1(a)|$. Since $|f_1(a)|=|g_1(a)|$, we see that $f_1(a)=g_1(a)$, so $f_1=g_1$. Now, $f_1$ is injective on $\mathscr{A}^+$ by Lemma~\ref{lemH}, so $f_2\circ f_3\circ \dots \circ f_m(a)=g_2\circ g_3\circ \dots \circ g_m(a)$. It follows from the inductive hypothesis that $f_i=g_i$, for $i=2,3,\dots,m$, as required.
\end{proof}

Now we prove the second part of Theorem~\ref{thmF}.

\begin{proof}[Proof of Theorem~\ref{thmF}: Part 2] 
Recall that $r>1$ and $s\leqslant |\mathscr{A}|^{r-1}$. Observe that there are $|\mathscr{A}|^{r-1}$ words over $\mathscr{A}$ of length $r$ with first letter $a$. Consequently, we can choose a set $\mathscr{F}$ of $s$ substitutions that satisfy $f[a]=a$, $\lvert f(a)\rvert=r$, and $f(a)\neq g(a)$, for each $a\in\mathscr{A}$ and each distinct pair $f,g\in\mathscr{F}$. Let $S$ be the substitution semigroup generated by $\mathscr{F}$, with forward limit set $\Lambda$.

We now define a probability measure $\mu$ on $\Lambda$ using the following standard procedure. Let $\mathscr{X}_n$ be the cover of $\Lambda$ introduced earlier (for $n=1,2,\dotsc$) and let $\mathscr{X}_0=\mathscr{A}^\mathbb{N}$. Observe that, by Lemma~\ref{lemO}, any two members of $\mathscr{X}_n$ are disjoint. Also, $\lvert \mathscr{X}_n\rvert =\lvert\mathscr{A}\rvert s^n$. Each set $U$ in $\mathscr{X}_n$ is contained in one set from $\mathscr{X}_{n-1}$ and contains $s$ sets from $\mathscr{X}_{n+1}$ The diameter of $U$ is $1/2^{r^n}$. Observe that the intersection of any nested sequence of sets $(U_n)$, with $U_n\in \mathscr{X}_n$, comprises a single infinite word in $\mathscr{A}^\mathbb{N}$. 

For $U$ in $\mathscr{X}_n$, we define
\[
\mu(U) = \frac{1}{\lvert\mathscr{A}\rvert s^n}
\]
and $\mu(E)=0$, where $E$ is the complement in $\mathscr{A}^\mathbb{N}$ of the union of all the members of $\mathscr{X}_n$. By \cite[Proposition~1.7]{Fa2003}, $\mu$ can be extended to a probability measure on $\mathscr{A}^\mathbb{N}$ that is supported on
\[
\Delta = \bigcap_{n=1}^\infty \bigcup_{U\in \mathscr{X}_n} U.
\]
Note that $\Delta \supseteq \Lambda$, since $\mathscr{X}_n$ is a cover of $\Lambda$. Conversely, if $x\in \Delta$, then there is $a\in\mathscr{A}$ and $w_n\in\mathscr{F}^n$, for $n=1,2,\dotsc$, with $F(w_n)(a)\to x$. Hence $x\in \Lambda$. Therefore $\Delta=\Lambda$, so $\mu$ is supported on $\Lambda$.

Suppose now that $(U_n)$ is any $\varepsilon$-cover of $\Lambda$, where $\varepsilon=1/2^r$. Consider some particular set $U$ in this cover, and let $k$ be the positive integer such that 
\[
\frac{1}{2^{r^{k+1}}}\leqslant \diam U < \frac{1}{2^{r^k}}.
\]
Then $U$ can intersect at most one of the sets from $\mathscr{X}_k$. Consequently,
\[
\mu(U) \leqslant \frac{1}{\lvert\mathscr{A}\rvert s^k}.
\]
Also,
\[
\psi_t(\diam U) \geqslant \psi_t\big(1/2^{r^{k+1}}\big)=\frac{1}{(r\log 2)^tr^{kt}}.
\]
Let $t=\log_rs$. Then 
\[
\psi_t(\diam U) \geqslant\frac{1}{(r\log 2)^ts^k}.
\]
Combining these inequalities gives
\[
\mu(U) \leqslant c\,\psi_t(\diam U),\qquad\text{where }c=\frac{(r\log 2)^t}{\lvert\mathscr{A}\rvert}.
\]
Thus
\[
\sum_{n=1}^\infty \psi_t(\diam U_n) \geqslant \frac{1}{c}\sum_{n=1}^\infty \mu(U_n) \geqslant \frac{1}{c} \mu \mleft(\bigcup_{n=1}^\infty U_n\mright)
>0.
\]
Thus $H^{\psi_t}(\Lambda)>0$, so $\Lambda$ has logarithmic Hausdorff dimension at least $t=\log_rs$. By applying the first part of the proof we see that in fact $\Lambda$ has logarithmic Hausdorff dimension exactly $\log_rs$, as required.
\end{proof}

\section{Hull of a substitution semigroup}\label{secJ}

Central to the study of substitution dynamics is the concept of the hull of a substitution. Here we introduce the hull of a substitution semigroup and characterise it by its invariant properties.

First we define the shift map $\sigma$, which satisfies $\sigma(a_1a_2\dotsc)=a_2a_3\dotsc$. More precisely, $\sigma$ is the self-map of $\widetilde{\mathscr{A}}$ that satisfies
\[
\pi_k(\sigma(w))=\pi_{k+1}(w),
\]
for $k\in\mathbb{N}$ (with $k<|w|$ if $w$ is finite). This map is continuous.

\begin{definition}
The \emph{hull} of the letter $a$ of $\mathscr{A}$ for the substitution semigroup $S$ is the set
\[
\Omega(a)=\overline{\bigcup_{n=0}^{\infty}\sigma^n(\Lambda(a))}\;.
\]
The \emph{hull}  of $S$ is the set $\bigcup_{a\in\mathscr{A}} \Omega(a)$. We denote this set by $\Omega$.
\end{definition}

The hull $\Omega(a)$ is closed and shift invariant, as can easily be verified. We will prove shortly that it is $S$-invariant. For this we need the next lemma.

\begin{lemma}\label{lemP}
Let $h$ be a substitution of an alphabet $\mathscr{A}$ and let $y$ be an infinite word over $\mathscr{A}$. Then for any positive integer $k$ there exists a positive integer $\ell$ for which
\[
h(\sigma^k(y))=\sigma^\ell(h(y)).
\]
\end{lemma}
\begin{proof}
We can write $y=xz$, where $x$ is a finite word of length $k$ and $z$ is an infinite word. Let $\ell=|h(x)|$. Then 
\[
\sigma^\ell(h(y))=h(z)=h(\sigma^k(y)),
\]
as required.
\end{proof}

Now we can prove that every hull $\Omega(a)$ is $S$-invariant.

\begin{lemma}\label{lemQ}
Let $S$ be a substitution semigroup of an alphabet $\mathscr{A}$ and let $a\in\mathscr{A}$. The hull $\Omega(a)$ is $S$-invariant.
\end{lemma}
\begin{proof}
Let $x\in\bigcup_n \sigma^n(\Lambda(a))$. Then $x\in\sigma^k(\Lambda(a))$, for some non-negative integer $k$. Hence there exists a sequence $(F_n)$ in $S$ and a word $y\in\Lambda(a)$ such that $F_n(a)\to y$ and $\sigma^k(y)=x$.

Let $h\in S$. By Lemma~\ref{lemP}, there is a positive integer $\ell$ such that
\[
h(\sigma^k(y))=\sigma^{\ell}(h(y)).
\]
That is, $h(x)=\sigma^{\ell}(h(y))$. Now, we know that $h(y)\in \Lambda(a)$, because $\Lambda(a)$ is $S$-invariant. Therefore $h(x)\in\sigma^{\ell}(\Lambda(a))$. 

This argument shows that $\bigcup_n \sigma^n(\Lambda(a))$ is $S$-invariant. By continuity of the action of $S$, the set $\Omega(a)$ is also $S$-invariant.
\end{proof}

We describe a substitution semigroup $S$ as \emph{irreducible} if, for any two letters $a$ and $b$ of $\mathscr{A}$, there exists an element $f$ of $S$ for which the word $f(a)$ contains the letter $b$. A subset $X$ of $\mathscr{A}^\mathbb{N}$ is said to be \emph{shift invariant} if $\sigma(X)\subseteq X$.

\begin{lemma}\label{lemT}
Let $S$ be an irreducible substitution semigroup and let $X$ be a closed subset of $\mathscr{A}^\mathbb{N}$ that is both $S$-invariant and shift invariant. Then $X$ contains $\Omega(a)$, for any $a\in\mathscr{A}$.
\end{lemma}
\begin{proof}
Choose $x\in X$. Since $S$ is irreducible, there is a substitution $h\in S$ for which the word $h(x)$ contains the letter $a$. Hence there is a non-negative integer $k$ for which $\pi_1(y)=a$, where $y=\sigma^{k}(h(x))$. Observe that $y\in X$, by $S$-invariance and shift invariance of $X$. 

Now choose $u\in \Lambda(a)$. Then there exists a sequence $(F_n)$ in $S$ with $F_n(a)\to u$. By applying Lemma~\ref{lemK} we see that $F_n(y)\to u$ also. Then $u\in X$ by $S$-invariance and closure of $X$. Consequently $\Lambda(a)\subseteq X$. Therefore $\Omega(a)\subseteq X$, by shift invariance, as required.
\end{proof}

An immediate corollary of Lemma~\ref{lemT} is that there is only one hull for an irreducible substitution semigroup.

\begin{corollary}\label{corD}
Let $S$ be an irreducible substitution semigroup. Then $\Omega(a)=\Omega(b)$, for all $a,b\in\mathscr{A}$.
\end{corollary}
\begin{proof}
The hull $\Omega(b)$ is closed and shift invariant, and it is $S$-invariant by Lemma~\ref{lemQ}. Hence $\Omega(a)\subseteq \Omega(b)$, by Lemma~\ref{lemT}. The roles of $a$ and $b$ can be reversed to give $\Omega(a)=\Omega(b)$.
\end{proof}

Therefore the hull $\Omega$ of an irreducible substitution semigroup $S$ is equal to $\Omega(a)$, for any $a\in\mathscr{A}$.

We finish with another corollary of Lemma~\ref{lemT}, which characterises the hull in terms of the actions of $S$ and $\sigma$ on $\mathscr{A}^{\mathbb{N}}$.

\begin{theorem}\label{thmK}
The hull of an irreducible substitution semigroup $S$ is the smallest closed subset of $\mathscr{A}^{\mathbb{N}}$ that is both $S$-invariant and shift invariant.
\end{theorem}
\begin{proof}
We know that the hull $\Omega$ is a closed subset of $\mathscr{A}^{\mathbb{N}}$ that is shift invariant, and it is $S$-invariant by Lemma~\ref{lemQ}. 

Now let $X$ be any closed subset of $\mathscr{A}^\mathbb{N}$ that is both $S$-invariant and shift invariant. Then $X$ contains $\Omega(a)$, by Lemma~\ref{lemT}, and $\Omega(a)=\Omega$, by Corollary~\ref{corD}, so $\Omega$ is the smallest closed subset of $\mathscr{A}^{\mathbb{N}}$ that is both $S$-invariant and shift invariant.
\end{proof}

\end{document}